\documentclass[11pt]{amsart}

\usepackage{amsmath}
\usepackage{url}
\usepackage{bm,amssymb,amscd}
\usepackage{graphicx,epsfig,color}
\usepackage[letterpaper, asymmetric, left=1.1in, right=1.1in, top=1.2in, bottom=1in, bindingoffset=0.0in]{geometry}

\newtheorem{theorem}{\bf Theorem}[section]
\newtheorem{proposition}[theorem]{\bf Proposition}
\newtheorem{definition}[theorem]{\bf Definition}
\newtheorem{lemma}[theorem]{\bf Lemma}

\newtheorem{problem}{\bf Problem}
\newtheorem*{remark}{\bf Remark}
\newtheorem*{maskittheorem}{\bf Maskit Combination Theorem}

\title{Problem on Mutant Pairs of Hyperbolic Polyhedra}

\begin{author}[C. Gyurek]{Croix Gyurek}
        \address{ %
                IUPUI Department of Mathematical Sciences\\
                LD Building, Room 255\\
                402 North Blackford Street\\
                Indianapolis, Indiana 46202-3267\\
                United States }
\end{author}

\begin{author}[R. K. W. Roeder]{Roland K. W. Roeder$^\dag$}
\thanks{$^\dag$Corresponding Author.  Email: \url{roederr@iupui.edu}}
        \address{ %
                IUPUI Department of Mathematical Sciences\\
                LD Building, Room 224Q\\
                402 North Blackford Street\\
                Indianapolis, Indiana 46202-3267\\
                United States }
\end{author}

\date{\today}
\begin{document}
\begin{abstract}We present a notion of mutation of hyperbolic polyhedra, analogous to mutation in knot theory \cite{CONWAY}, and then present a general question about commensurability of mutant pairs of polyhedra. We motivate that question with several concrete examples of mutant pairs for which commensurability is unknown. The polyhedra we consider are compact, so techniques involving cusps that are typically used to distinguishing mutant pairs of knots are not applicable.  Indeed, new techniques may need to be developed to study commensurability of mutant pairs of polyhedra.\end{abstract}
\maketitle
\section{Introduction}
Let $P$ be a compact polyhedron in hyperbolic 3-space $\mathbb{H}^3$. If the
dihedral angle measures are all integer submultiples of $\pi$, then reflections
in faces of $P$ induce a discrete group $\Lambda(P) \leq Isom(\mathbb{H}^3)$, $P$ serves as a fundamental domain for the action of $\Lambda(P)$, and $\mathbb{H}^3$ is tiled by images of $P$ under elements of $\Lambda(P)$; see, e.g. \cite[Theorem 6.4.3]{DAVIS}. In
this case $P$ is called a {\em Coxeter polyhedron}. There is a complete
classification of hyperbolic Coxeter polyhedra, given by Andreev's Theorem
\cite{ANDREEV,RHD} (see also \cite{HR,HOD,BS} for alternatives to the classical
proof).

The group $\Lambda(P)$ has an orientation-preserving index-2 subgroup
$\Gamma(P)$, which is therefore a Kleinian group, i.e. $\Gamma(P)$ is a discrete
subgroup of $PSL(2,\mathbb{C})$. Study of these polyhedral reflection groups
$\Gamma(P)$ is a classical topic in hyperbolic geometry \cite{VINBERG}.
They are also of considerable recent interest from several different perspectives; as a sample, we refer the reader
to \cite{AGOL,MISHA,DAVIS,GJK,LOR}.

\begin{definition}
A {\em 3-prismatic circuit} of a polyhedron $P$ is a triple of faces $(F_1, F_2, F_3)$ such that each face is adjacent to the other two, but the three faces do not all meet at a common vertex.
\end{definition}

\begin{definition}Two compact hyperbolic polyhedra $P$ and $P^\prime$ form a {\em mutant pair} iff $P^\prime$ can be obtained from $P$ by the following procedure: Start with a 3-prismatic circuit in $P$ such that the dihedral angles between any two of those faces are equal. Find the common plane $N$ that is perpendicular to all three faces. Split $P$ into two parts $P_1, P_2$ on either side of $N$, and rotate $P_2$ by $\frac {2\pi} 3$ (in whichever direction you choose). Then glue $P_1$ and the rotated $P_2$ back together to obtain $P^\prime$.
\end{definition}
In this paper we will build our polyhedra out of half-polyhedra that have a single plane of reflectional symmetry. This means that the direction of rotation is irrelevant.

As an example, consider the polyhedra depicted in Figure 1, and suppose the
vertical edges each have dihedral angle $\frac \pi 4$. The one on the right is
obtained by slicing the left polyhedron along the plane that intersects the prismatic circuit in the dashed triangle, rotating the lower half by $\frac {2\pi} 3$,
and gluing it back together.
\begin{figure}[h] \label{mutantexample}
\includegraphics[height=4cm]{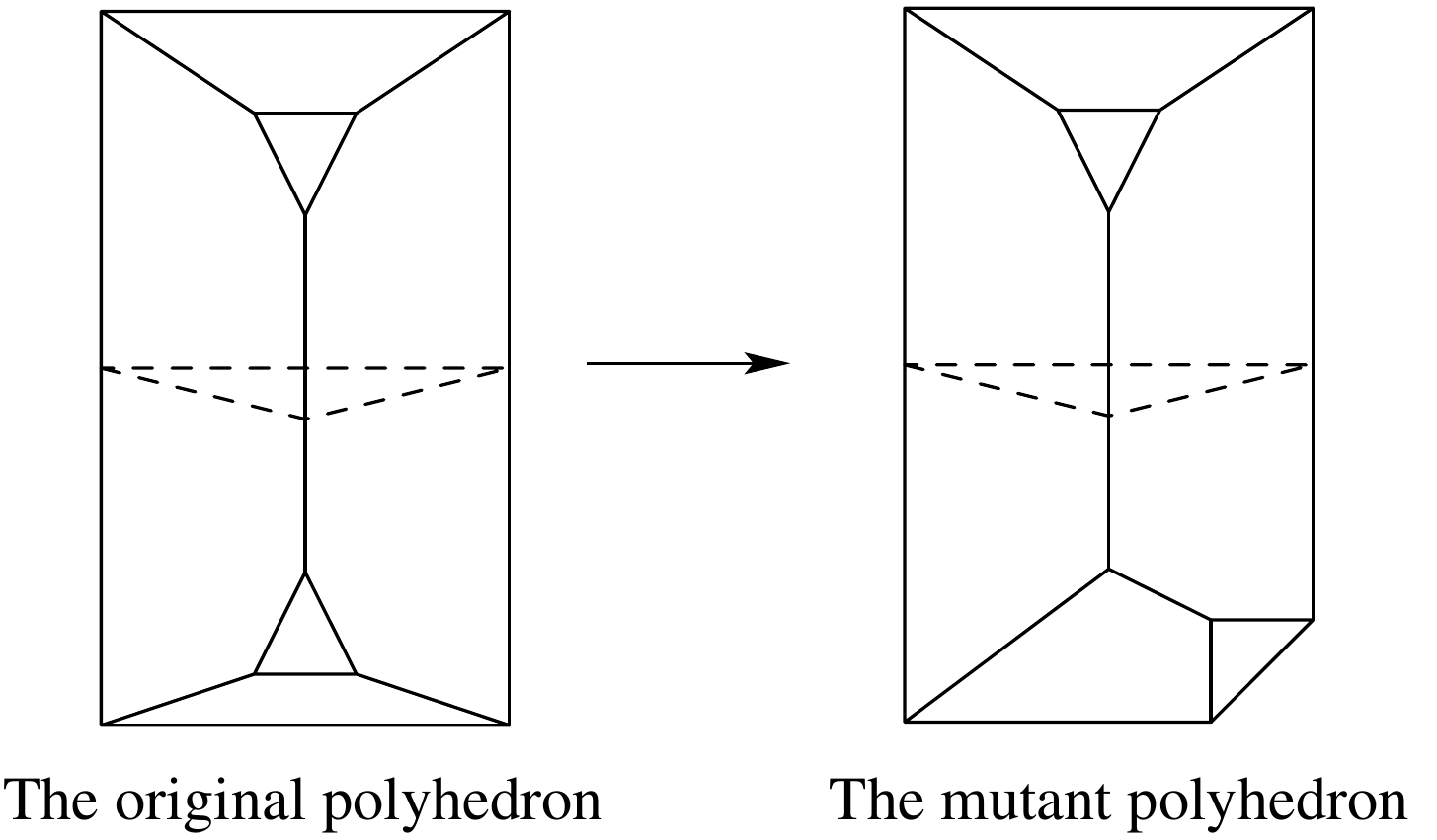}

\caption{Slicing the first polyhedron in half and rotating the lower half by $\frac {2\pi} 3$ results in the second.}
\end{figure}

\begin{definition}
Two Kleinian groups $\Gamma_1$ and $\Gamma_2$ are {\em commensurable (in the wide sense)} iff there exists an element $g \in PSL(2,\mathbb{C})$ such that $\Gamma_1$ and $g\Gamma_2g^{-1}$ have a common finite-index subgroup.

Two compact hyperbolic polyhedra $P,P^\prime$ are {\em commensurable} iff their Kleinian groups $\Gamma(P)$ and $\Gamma(P^\prime)$ are commensurable in the wide sense.
\end{definition}

Pairs of commensurable Kleinian groups share several properties, and therefore determining if a pair of Kleinian groups is commensurable is an important problem. We refer the reader to the textbook by Maclachlan and Reid and the references therein for details \cite{MRbook}. The most common invariants of commensurability are the invariant trace field (ITF) and invariant quaternion algebra (IQA); see Section~2 for more details.

For essentially the same reasons as for mutation of knots, mutation of hyperbolic polyhedra preserves the ITF and usually the IQA of the Kleinian group (this will be explained in Section~2). Therefore determining whether a mutant pair is commensurable is a difficult problem.

\begin{problem}
Under what conditions is a mutant pair of hyperbolic polyhedra commensurable?
\end{problem}

\begin{remark}
For mutant pairs of knots, the knots serve as cusps at infinity for their
complements.  These cusps lead to additional invariants, which often allow one
to determine whether the two knot complements are commensurable; see for
example \cite{CUSPS} and \cite{CD}.  
In this note we have restricted our
attention to compact polyhedra so that techniques related to cusps cannot be used
on Problem 1.
\end{remark}

In order to show that Problem 1 is interesting, we present several very simple
examples of mutant pairs of polyhedra in Sections~3 and 4. We use the ``SNAP-HEDRON''
software \cite{SNAP,SNAPHEDRON} to study these pairs of polyhedra. In certain
cases, one group is arithmetic and the other is not, or one group has integral
traces and the other does not; in both cases, the pair of groups are not commensurable.
But in most cases, the question of commensurability is unknown.
(Remark that mutant pairs BB4, BB4m and BB5, BB5m were studied in \cite[Section 7.6]{SNAP}
and the purpose of the present paper is to amplify that discussion and produce many additional examples.)

\begin{problem}
Under what conditions on a Coxeter polyhedron $P$ does its reflection group $\Gamma(P)$ have integral traces?
\end{problem}

Indeed, if one looks at Table 1, many of the CC family of polyhedra have integral traces. (Determining which Coxeter polyhedra result in arithmetic reflection groups is a very classical problem,
so we state only the variant about integral traces in Problem 2, above.)

\vspace{0.1in}
\noindent
{\bf Acknowledgments:}
We thank Patrick Morton and Daniel Ramras for helpful discussions.
This work was supported by NSF grant DMS-1348589.

\section{Mutation preserves some commensurability invariants}
\begin{definition}Let $\Gamma$ be a non-elementary finitely generated Kleinian group. Consider the subgroup $\Gamma^{(2)} = \langle \gamma^2 : \gamma \in \Gamma \rangle$.

The {\em invariant trace field (ITF)} $k\Gamma$ of $\Gamma$ is the field 
$$\mathbb{Q}\left(\text{tr } \Gamma^{(2)}\right)$$

The {\em invariant quaternion algebra (IQA)} $A\Gamma$ of $\Gamma$ is defined as
$$A\Gamma := \left\{\sum a_i \gamma_i : a_i \in \mathbb{Q}\left(\rm{tr}\: \Gamma^{(2)}\right), \gamma_i \in \Gamma\right\}.$$
\end{definition}

The following two propositions show that mutation of hyperbolic polyhedra preserves the ITF, and sometimes the IQA, of their corresponding reflection groups. The examples that we present in Sections 3 and \ref{SEC:EXAMPLES} have been constructed to preserve both.

\begin{proposition} \label{ITFsame}
If $P$ and $P^\prime$ are a mutant pair, then $k\Gamma(P) = k \Gamma(P^\prime)$.
\end{proposition}
To prove this we need the following two results:

\begin{theorem}[\bf McLaughlin-Reid {\cite[Theorem 5.6.1]{MRbook}} ] \label{MR561}
If $\Gamma$ is a finitely generated Kleinian group expressed as a free product with amalgamation $\Gamma_1 *_H \Gamma_2$, where H is a non-elementary Kleinian group, then $k\Gamma = k\Gamma_1 \cdot k\Gamma_2$, where $\cdot$ denotes the compositum.
\end{theorem}
\begin{maskittheorem} \label{Maskit}
Let $\Gamma_1,\Gamma_2$ be two Kleinian groups with $\Gamma_1 \cap \Gamma_2 = H$, where $H$ is quasi-Fuchsian with limit set $L_H$. If $S^2 \, \backslash \, L_H$ consists of two disjoint parts $B_1 \cup B_2$, and for $i \in \{1,2\}$ only those $g_i \in \Gamma_i$ with $g_iB_i \cap B_i$ nonempty are those in H, then $\Gamma = \langle \Gamma_1 , \Gamma_2 \rangle \cong \Gamma_1 *_H \Gamma_2$, and $\Gamma$ is Kleinian.
\end{maskittheorem}
\begin{proof}[Sketch of proof of Proposition \ref{ITFsame}]
Let $P$ and $P^\prime$ be a mutant pair of polyhedra with cutting plane $N$ and
half-polyhedra $P_1$ and $P_2$. According to Theorem \ref{MR561}, it suffices to
prove that for a suitable choice of Kleinian groups $\Gamma_1,\Gamma_2,$ and
$H$, we have $\Gamma(P) \cong \Gamma_1 *_H \Gamma_2$ and also 
$\Gamma(P^\prime) \cong \Gamma_1 *_H \Gamma_2$, possibly with different
amalgamation maps.  This will follow using the Maskit combination theorem. 

Let us provide the details that $\Gamma(P) \cong \Gamma_1 *_H \Gamma_2$ because
proof for $\Gamma(P')$ is identical.  It will be helpful to suppose $P$ is
situated in the Poincar\'e ball model with the each of the faces of the
prismatic circuit perpendicular to the equatorial plane; i.e.\ the cutting
plane $N$ is the equatorial plane $z=0$.    
Let
\begin{align*}
H_+:= \{(x,y,z) \in \mathbb{H}^3 \, : \, z > 0\} \qquad \mbox{and} \qquad  H_-:= \{(x,y,z) \in \mathbb{H}^3 \, : \, z < 0\}.
\end{align*}
We also suppose $P_1$ is in the closed upper half space $\overline{H_+}$
and $P_2$ is in the closed lower half space $\overline{H_-}$.

It will be useful to consider infinite volume analogs of the half polyhedra
$P_1$ and $P_2$, so we let $P_1^0$ be the polyhedron bounded by the faces
of $P_1$ other than $N$, and define $P_2^0$ analogously.   Let $T$ be the
polyhedron bounded by the faces of the prismatic $3$-circuit $C$.

Let the groups $\Lambda_1, \Lambda_2$, and $J$ be the discrete subgroups of
${\rm Isom}(\mathbb{H}^3)$ generated by reflections in $P_1^0$, $P_2^0$, and
$T$, respectively.  They are subgroups of the reflection group $\Lambda(P)$. 

\vspace{0.1in}
\noindent
{\em Claim:} If $\lambda \in \Lambda_1$ with $\lambda(H_-) \cap H_- \neq \emptyset$ then $\lambda \in J$.  (The analogous statement
holds for $\Lambda_2$ with $H_-$ replaced by $H_+$.)

\vspace{0.1in}
\noindent
{\em Proof of Claim:}
As $\Lambda_1$ is a discrete group with fundamental domain $P_1^0$, the union of all images of $P_1^0$ under $\Lambda_1$ tile $\mathbb{H}^3$
with multiplicity one.  Moreover, the union of all images of $P_1^0 \cap H_-$ under $J$ tile $H_-$ with multiplicity one.  (This is because $T$ is a fundamental
domain for $J$ and $J$ preserves both halfspaces.)  Therefore, if $\lambda \in \Lambda_1$ satisfies $\lambda(H_-) \cap H_- \neq \emptyset$
it must be an element of $J$, since otherwise an open set of points in $\mathbb{H}^3$ would be double-tiled.
\qed (Claim)

\vspace{0.1in}

The claim immediately implies that $\Lambda_1 \cap \Lambda_2 = J$.  If we let
$B_1 \subset \mathbb{S}^2$ be the lower open hemisphere and $B_2$ be the upper
open hemisphere, the claim also proves that if $\lambda_i \in \Lambda_i$ with
$\lambda_i(B_i) \cap B_i \neq \emptyset$ then $\lambda_i \in J$.  
Since $\Gamma_1$, $\Gamma_2$, and $H$ are subgroups of $\Lambda_1$, $\Lambda_2$, and $J$ respectively, the orientation-preserving groups satisfy the conditions of the Maskit Combination Theorem. Therefore, $\langle \Gamma_1 , \Gamma_2 \rangle \cong \Gamma_1 *_H \Gamma_2$.

It remains to check that $\langle \Gamma_1, \Gamma_2 \rangle = \Gamma(P)$.  The containment  $\langle \Gamma_1, \Gamma_2 \rangle \subset \Gamma(P)$
is immediate, so we will show that $\Gamma(P) \subset \langle \Gamma_1, \Gamma_2 \rangle$.
Let $\gamma = r_1 r_2 \cdots r_{2n-1} r_{2n} \in \Gamma(P)$ with each $r_i \in \Lambda(P)$.  Let $s \in J = \Lambda_1 \cap \Lambda_2$.  Then
\begin{align*}
\gamma = r_1 r_2 r_3 r_4 \cdots r_{2n-1} r_{2n} = (r_1 s) (s r_2) (r_3 s) (s r_4) \cdots (r_{2n-1} s) (s r_{2n}) \in \langle \Gamma_1, \Gamma_2 \rangle.
\end{align*}
\end{proof}

\begin{proposition} \label{IQAsame}
If $P$ and $P^\prime$ are a mutant pair of hyperbolic polyhedra and there exists a vertex of $P$ (and hence also $P^\prime$) with two edges meeting at it having dihedral angles $\frac \pi n$ and $\frac \pi m$, with $m,n \geq 3$, then $A(P) \approx A(P^\prime)$.
\end{proposition}
\begin{proof}
If two edges meeting at a vertex have angles $\frac \pi n$ and $\frac \pi m$, with $m,n \geq 3$, then the local group of that vertex is either $T_{12}, O_{24},$ or $I_{60}$. See, for example, \cite{BMP}, Figure 3 in Chapter 2. Each of the groups $T_{12},O_{24}, {\rm and}\, I_{60}$ contains a subgroup isomorphic to $A_4$.  Therefore, by Lemma~5.4.1 in \cite{MRbook}, $A(P)$ is isomorphic to an algebra depending only on $k(P)$, which is isomorphic to $k(P^\prime)$ by Proposition \ref{ITFsame}.
\end{proof}

The next invariant property allows us to prove that four of the mutant pairs of polyhedra that are presented in Section 4 are not commensurable.

\begin{definition}
We say that $\Gamma$ has {\em integral traces} iff for all $\gamma \in \Gamma, \text{tr}\, (\gamma)$ is an algebraic integer.
\end{definition}

\begin{proposition} \label{inttracec} The property of having integral traces is a commensurability invariant.
\end{proposition}
\begin{proof} See Exercise 1 in Section 5.2 of \cite{MRbook}. \end{proof}

This next lemma makes it easy to check if a reflection group has integral traces.
\begin{lemma} \label{alginttrace}
If $P$ is a Coxeter polyhedron with Gram matrix G, then $\Gamma(P)$ has integral traces iff every element of $G$ is an algebraic integer.
\end{lemma}
\begin{proof}
This is shown on page 325 of \cite{MRbook} (which uses $\Gamma^+(P)$ for the orientation-preserving subgroup).
\end{proof}

In certain cases, reflection groups will be arithmetic. If that holds, Theorem 8.3.2 from \cite{MRbook} implies that the group has integral traces.

\section{Simple example for which commensurability is unknown.}

We showcase Problem 1 by presenting a very simple mutant pair of polyhedra for which we do not know commensurabilty.

\begin{problem}
Are the two prisms shown in Figure \ref{FIG_SIMPLE} commensurable?
\end{problem}

\begin{figure}[h]
\includegraphics[height=5cm]{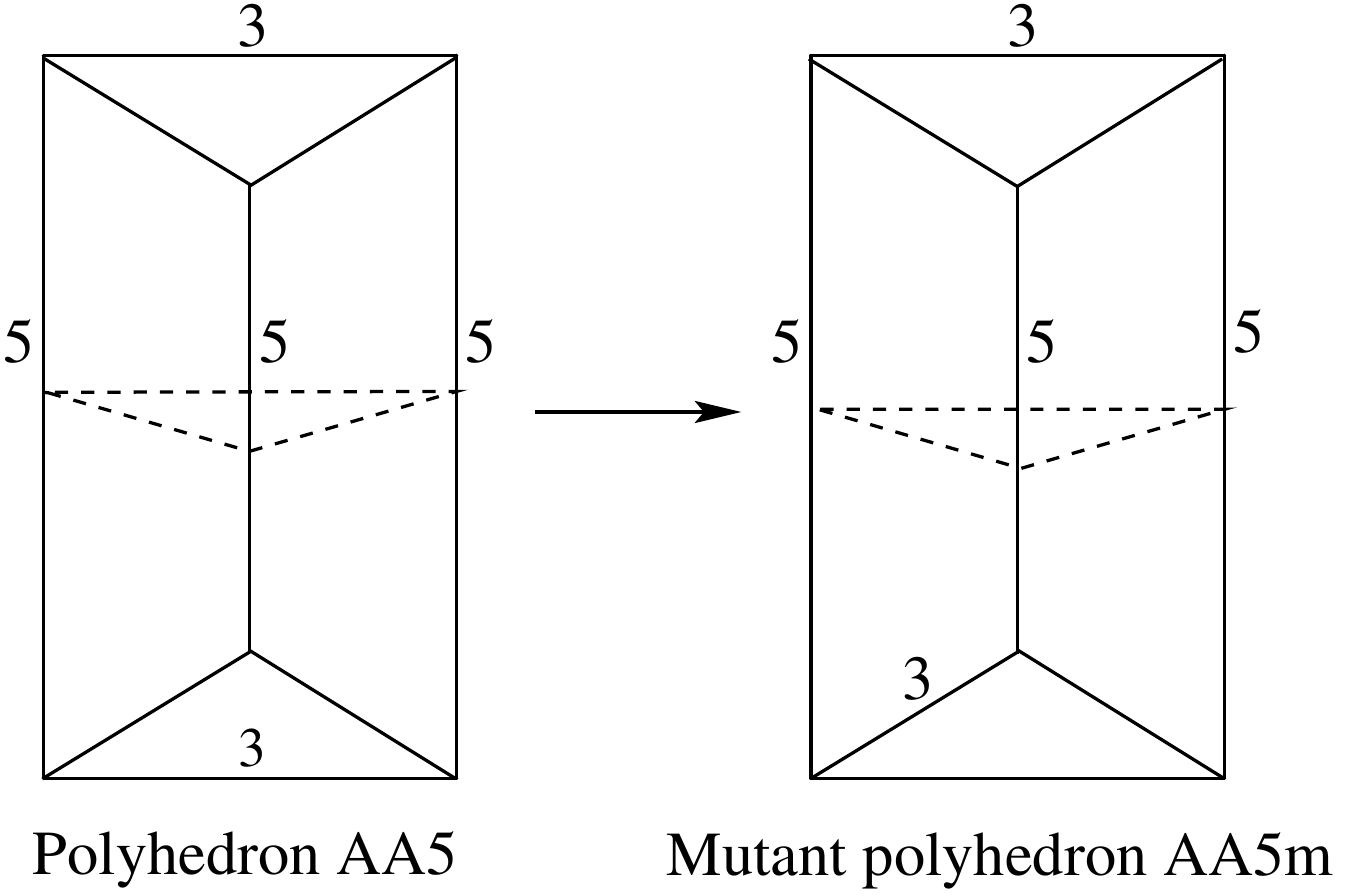}
\caption{\label{FIG_SIMPLE} A mutant pair of prisms for which we do not know if they are commensurable or not.  An edge is labeled by $n$ if it is assigned a dihedral angle of~$\frac \pi n$.  Unlabeled edges are assigned right dihedral angles.}
\end{figure}

\noindent
We will refer to the two polyhedra in Figure \ref{FIG_SIMPLE} as AA5 and AA5m, respectively, to be consistent
with the notation of Section 4.
By Proposition \ref{ITFsame}, both 
AA5 and AA5m have the same invariant trace field, which in this case is 
\begin{align}\label{EXAMPLE_ITF}
k\Gamma(P) = \mathbb{Q}\left(\frac{1}{2}-\frac{i}{2}\sqrt {-5+8\,\sqrt {5}}\right).
\end{align}
Moreover, each of the polyhedra has a vertex where two edges having
dihedral angles $\frac \pi 3$ and $\frac \pi 5$ meet.  Therefore,
Proposition \ref{IQAsame} gives that AA5 and AA5m have isomorphic invariant quaternion algebras, each of them satisfying
\begin{align*}
A\Gamma(P) \approx \left(\frac{-1,-1}{k\Gamma(P)}\right).
\end{align*}
Moreover, both have non-integral traces.  For these reasons it is difficult to
determine if the polyhedra AA5 and AA5m are commensurable.

We used the SNAP-HEDRON software\footnote{It could probably also be done by hand, like in \cite[Section 4.7.3]{MRbook}.} to the compute the outward pointing normal vectors for these polyhedra and also their Gram Matrices.
  Let us record the details here, partly to show how concrete Problem 3 is.
We list the exact outward pointing normal vectors as rows of the following matrices. Each
row represents a vector in $E^{3,1}$ with the timelike coordinate first.

$$N_{\rm AA5} = \begin{bmatrix}
0 & 0 & 0 & 1 \\

0 & 0 & -1 & 0 \\

0 & -\frac 1 4 \sqrt {10-2 \sqrt 5} & \frac {1 + \sqrt 5} 4 & 0 \\

-\frac 1 2 \sqrt {6\sqrt {5}+11} & \frac 1 4\sqrt {50+22\sqrt {5}} & \frac {1 + \sqrt {5}}4  & -\frac 1 2 \\
-\frac 1 {20}\sqrt {-130+90\sqrt {5}} & 0 & 0 & -\frac 3 4-{\frac {3\,\sqrt {5}}{20}} 
\end{bmatrix}, \quad \mbox{and}$$ 
$$N_{\rm AA5m} = \begin{bmatrix}
0 &        0 &           0 &              1 \\
0 &        0 &           -1 &              0 \\
0 & -\frac{1}{4}\,\sqrt {10-2\,\sqrt {5}} & \frac {1 + \sqrt {5}}4 & 0 \\
-\frac 1 2 \sqrt {6\sqrt {5}+11} & \frac 1 4\sqrt {50+22\sqrt {5}} & \frac {1 + \sqrt {5}}4  & -\frac 1 2 \\
-\frac{1}{10}\,\sqrt {55+30\,\sqrt {5}} & \frac{1}{10}\,\sqrt {50+10\,\sqrt {5}} & 0 & -1-\frac{1}{10}\,\sqrt {5} \\
\end{bmatrix}. \qquad \,\,$$
Notice that the only difference occurs in the last outward pointing normal
vector (last row of the matrix), which corresponds to the bottom face in Figure
\ref{FIG_SIMPLE}.

The Gram matrices\footnote{We use the convention from \cite[Section
10.4]{MRbook} so that $G_{i,j} = 2\left<{\bm v}_i,{\bm v}_j \right>$.} 
of the polyhedra are
$$G_{\rm AA5} = 
\begin{bmatrix}
2      & 0     & 0     & -1 &  \alpha    \\
0      & 2     & -\phi & -\phi      & 0 \\
0      & -\phi & 2     & -\phi     & 0 \\
-1 &  -\phi     &  -\phi    & 2      & -1     \\
\alpha   & 0 & 0 & -1      & 2
\end{bmatrix} \quad  \mbox{and} \quad
G_{\rm AA5m} =
\begin{bmatrix}
2      & 0     & 0     & -1 &  \beta    \\
0      & 2     & -\phi & -\phi      & 0 \\
0      & -\phi & 2     & -\phi     & -1 \\
-1 &  -\phi     &  -\phi    & 2      & 0     \\
\beta   & 0 & -1 & 0      & 2
\end{bmatrix}
$$
where, $\alpha = -\frac{3}{2}-\frac{3\sqrt{5}}{10}$, $\beta = -2 - \frac
{\sqrt 5} 5$, and $\phi = 2 \cos\left(\frac{\pi}{5}\right) = \frac {1+\sqrt 5}2$ is the golden ratio.
One can then use \cite[Theorem 10.4.1]{MRbook} to compute the invariant trace field (\ref{EXAMPLE_ITF}).

\section{Several Additional Mutant Pairs}\label{SEC:EXAMPLES}
Now we will present a collection of several additional simple mutant pairs of polyhedra, to demonstrate Problem~1. Of the 15 pairs presented in Table 1, four have been shown to not be commensurable; the other 11 are open. We constructed the polyhedra using the HYPER-HEDRON software \cite{HYPERHEDRON} for MATLAB, and then computed the invariants and checked for arithmeticity and/or integral traces using a PARI/GP program called SNAP-HEDRON \cite{SNAPHEDRON}.

To this end, we define three ``half-polyhedra'' A,B, and C as follows. The
numbers $n$ on the edges mean a dihedral angle of $\frac \pi n$; 
unlabeled edges are right-angled. By Andreev's Theorem, there is a
unique compact polyhedron realizing types A and B for $q \in \{4,5\}$, and type C for any $q \geq 4$. We interpret these as compact polyhedron with the dashed triangle as a face. When we form our mutant pairs, the dashed face will be where the gluing is done. Indeed, given any two of these half-polyhedra with the same choice of $q$, we can glue them together along the dashed plane to form a compact polyhedron in such a way that the faces above
and below a dashed edge ``merge'' to become a single face in the resulting
polyhedron.

\begin{figure}[h]
\includegraphics[height=4cm]{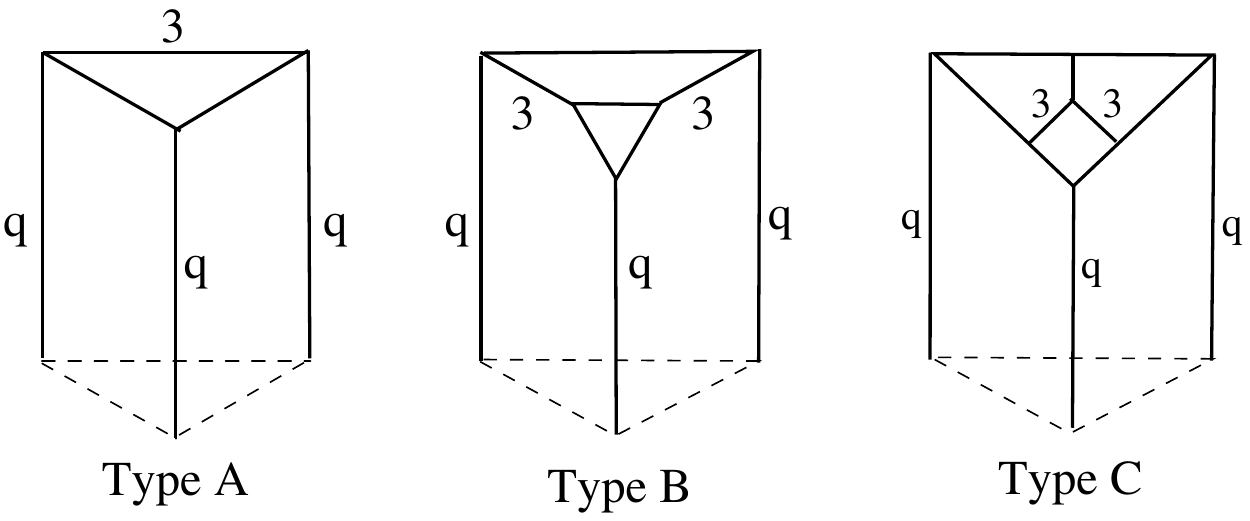}

\caption{The definitions for the three half-polyhedra. Unlabeled edges have right dihedral angles.}
\end{figure}

We use the following notation to describe these polyhedra. The first two
letters refer to the two halves that make it up (like AC or BB). Then we put a
number $q$ to denote that each of the edges of the prismatic circuit where the
mutation occurs has dihedral angle $\pi/q$. Finally, we append the letter ``m''
to refer to the mutated version. For example, Figure 1 describes BB4 (left) and
BB4m (right).

Note that all of the halves have a single plane of reflectional symmetry; the
``mutated'' polyhedra are those where the symmetry planes of the halves do not
coincide. (This means it does not matter whether we rotate clockwise or
counter-clockwise to get the mutation.) We also chose the dihedral angles so
that mutation preserves the IQA by Proposition \ref{IQAsame}.

SNAP-HEDRON data files for these polyhedra are available at: \\ \url{https://www.math.iupui.edu/~roederr/SNAP-HEDRON/mutant_pairs/}.


%
\begin{table}[!htbp]
\begin{tabular}{|l l|l|}
\hline
Polyhedra & Comments & Common ITF \& Approx. Root \\
\hline
\textbf{AA4} & Arithmetic & $x^4 - 2x^2 - 4x - 2$  \\
\textbf{AA4m} & Non-integral trace & $-0.7071067811865475 + 0.9561451575849219i$  \\
\hline
AA5 &   & $x^4 - 2x^3 - x^2 + 2x - 19$ \\
AA5m &   & $0.5000000000000000 - 1.795030906419045i$ \\
\hline
AB4 &  & $x^8 - 4x^6 + 4x^4 - 4x^2 + 1$ \\
AB4m &  & $-0.8040190354753588 + 0.5946035575013605i$ \\
\hline
AB5 &  & $x^8 - 16x^6 + 76x^4 - 256x^2 + 256$ \\
AB5m &  & $1.697635417647177 + 1.057371263440564i$ * \\
\hline
\textbf{AC4} &Integral traces & $x^8 - 4x^7 + 4x^6 - 4x^5 + 12x^4 - 4x^3 - 12x^2 + 4x + 1$  \\
\textbf{AC4m} &Non-integral trace & $1.550501494807729 + 0.1007902622880403i$ \\
\hline
AC5 & & $x^{16} + 5x^{14} + 29x^{12} + 165x^{10} + 596x^8 + 1215x^6 + 1139x^4 - 20x^2 + 1$  \\
AC5m & & $0.1386040292958391 - 0.1007017218363281i$ \\
\hline
\textbf{BB4} &Arithmetic& $x^4 - 2$  \\
\textbf{BB4m} &Non-integral trace& $0.0000000000000000 - 1.189207115002721i$ \\
\hline
BB5 && $x^4 - 20$  \\
BB5m && $0.0000000000000000 + 2.114742526881128i$  \\
\hline
BC4 & & $x^{16} - 4x^{14} - 16x^{13} - 16x^{12} + 88x^{11} + 116x^{10} - 168x^9 - 196x^8$ \\
BC4m & &\;\;   $+\, 168x^7 + 116x^6 - 88x^5 - 16x^4 + 16x^3 - 4x^2 + 1$ \\
 & & $-0.8257720711453254 - 2.505381134230215i$ \\
\hline
BC5 & & $x^{16} + 6x^{14} + 15x^{12} + 14x^{10} + 9x^8 + 14x^6 + 15x^4 + 6x^2 + 1$ \\
BC5m & & $0.5253337654545300 - 1.616811081461693i$ \\
\hline
\textbf{CC4} &Integral traces& $x^8 - 4x^7 + 4x^6 - 4x^5 + 12x^4 - 4x^3 - 12x^2 + 4x + 1$ \\
\textbf{CC4m} &Non-integral trace& $-0.5505014948077293 - 1.452983711741997i$  \\
\hline
CC5 &Integral Traces& $x^8 - 2x^6 + 4x^4 - 3x^2 + 1$ \\
CC5m &Integral Traces& $1.029085513635746 + 0.7476743906106103i$ \\
\hline
CC6 && $x^8 - 2x^7 - 4x^6 + 16x^5 - 26x^4 + 32x^3 - 16x^2 - 16x + 16$ \\
CC6m && $1.329130528042722 - 0.4831273532153546i$ \\
\hline
CC7  &Integral Traces& $x^{12} - 6x^{10} + 8x^8 - 13x^6 + 8x^4 - 6x^2 + 1$ \\
CC7m &Integral Traces& $0.8128813154618254 + 0.5824293665098390i$ \\
\hline
CC8  &Integral Traces& $x^{16} - 8x^{14} + 8x^{12} - 8x^{10} + 18x^8 - 8x^6 + 8x^4 - 8x^2 + 1$ \\
CC8m &Integral Traces& $0.6554951243769758 + 0.75519940540099274i$ \\

\hline
\end{tabular}
\vspace{0.10in}
\caption{\label{TABLE1} Data for the 15 mutant pairs. Next to each polyhedron name is listed whether or not the reflection group is arithmetic or has integral traces. All others have non-integral traces. \textbf{Boldface} names refer to non-commensurable mutant pairs that are distinguished by Proposition \ref{inttracec}. The right column shows the generator polynomial and approximate root for the invariant trace field, which is identical for both polyhedra by Proposition \ref{ITFsame}.}
\vspace{-0.35in}
\end{table}

\newpage

\bibliographystyle{plain}
\bibliography{rrefs.bib}

\vspace{0.2in}

\end{document}